\documentclass{article}

\usepackage{bbm}
\usepackage{eufrak} 
\usepackage{amssymb}
\usepackage{amsmath}
\usepackage{amscd}  
\usepackage{amsthm}

\renewcommand{\contentsname}{Contents\\{\footnotesize\normalfont(A table
of contents should normally not be included)}}

\newcommand\Rs{{\mathbb R}}

\newcommand\Ss{{\mathbb S}}

\newcommand\dd{\partial}
\newcommand\MM{{\mathcal  M}} 
\newtheorem{lemma}{Lemma}
\newtheorem{theorem}{Theorem}
\newtheorem{corollary}{Corollary}

\begin{document}

\title{The Euler characteristic of a triangulated  manifold in terms of even-dimensional faces}
\author{Alexey V. Gavrilov}

\maketitle


\begin{abstract}
For a triangulated  manifold $\MM$ of even dimension $d$ we show that\\  
$\chi(\MM)=\sum_{m=0}^{d/2} c_m f_{2m}$ where $\chi(\MM)$
is the Euler characteristic and $f_n$ is  the number of $n$-dimensional faces in the  triangulation. The coefficients $c_m$ in this formula do not depend on $d$.
\end{abstract}

\section{Introduction} 

The Euler characteristic  of a finite simplicial complex $\MM$ is given by the classical formula
$$\chi(\MM)=\sum_{n=0}^\infty (-1)^n f_n.$$
On the right hand side are  $f$ - numbers of the complex,  that is  $f_n$ denotes the number  of faces of dimension $n$.    We are interested in the case when $\MM$ is  a triangulated closed  manifold 
where by a triangulation we mean a homeomorphism to a finite simplicial complex
(which is  not necessarily  a PL manifold). The Euler characteristic  of a closed manifold of odd dimension is zero so we only consider manifolds of even dimension $d=\dim\MM$.

In this situation the above is not the only possible formula for the Euler characteristic  because there exist linear relations between $f$ - numbers known as Dehn-Sommerville relations. (For example,  $f_{d-1}=\frac{d+1}{2}f_d$  as   any  $(d-1)$ - face is shared between  two $d$ - faces.)   In this note we prove a formula whose right hand side only includes even-dimensional $f$ - numbers. Of course,  in each particular dimension the problem is  trivial as long as the  relations are known: for example, if $\MM$ is a surface then $f_1=\frac{3}{2}f_2$ hence  $\chi(\MM)=f_0-\frac{1}{2}f_2$. The same way,   for a fourfold we have $f_3=\frac{5}{2}f_4$ and (less obviously) $2f_1-3f_2+4f_3-5f_4=0$, so  $\chi(\MM)=f_0-\frac{1}{2}f_2+f_1$. The Dehn-Sommerville relations for all  dimensions   were described by Victor  Klee \cite{K} which reduces  our task to a certain combinatorial problem.  

Let $\beta_n\,\,(n\ge -1)$
be  a sequence  defined by 
$$\beta_{n-2}=\frac{4(2^n-1)B_n}{n}$$
where $B_n$ are the Bernoulli numbers; note that $\beta_n=0$ when $n$ is odd with the exception of $\beta_{-1}=-2$. 
A semi-Eulerian complex is a pure simplicial complex such that  the link of any face of dimension $n$ has the same Euler characteristic $1-(-1)^{d+n}$ as the sphere  $\Ss^{d-n-1}$ \cite{K, NS, CM}  
(Klee called it an  Eulerian manifold).
$$$$
\begin{theorem}
If $\MM$ is a semi-Eulerian complex  of even dimension $d$ then 
\begin{equation}
\chi(\MM)=\sum_{n=0}^{d} \beta_n f_n. \label{t}
\end{equation}
\end{theorem}

It is known that  if a simplicial complex is  homeomorphic to a closed manifold then 
the link of any  face has the same homology groups as the sphere 
of the corresponding dimension. 
(Klee himself refers to  Lefschetz's ``Topology''.  Apparently,  the first modern proof of this fact  is given in \cite[Proposition 1.2]{GS}.) 
Consequently, it is always a semi-Eulerian complex. Today  it is known that some  topological manifolds  admit no  triangulations at all  but we do not discuss this interesting topic here.  

The formula itself is not particularly surprizing:  the Dehn-Sommerville relations reduce the dimension of the space of $f$ - vectors from $d+1$ to $\frac{d}{2}+1$, so it is natural to  expect 
for even $f$ - numbers to  be sufficient.   What is remarkable is that  the coefficients of this formula do not depend on the dimension $d$. (The author does not know if this fact has some  topological explanation or is merely a kind of coincidence.)  It should be  possible to generalize this result using more general Dehn-Sommerville relations  (see \cite{CM} and references therein) but in the end  the author decided to leave the formulation  as it is. On the other hand, extending the formula  to a manifold with a boundary is not difficult. 
$$$$
\begin{corollary}
If $\MM$ is an even-dimensional  PL manifold  with a boundary  then 
$$\chi(\MM)=\sum_{n=0}^{d} \beta_n (f_n(\MM)-\tfrac{1}{2}f_n(\dd\MM)).$$
\end{corollary}
Indeed, we have  
$$2\chi(\MM)=2\chi(\MM)-\chi(\dd\MM)=\chi(\widetilde{\MM})=\sum_{n=0}^{d} \beta_n f_n(\widetilde{\MM})=\sum_{n=0}^{d} \beta_n (2f_n(\MM)-f_n(\dd\MM))$$ 
where  $\widetilde{\MM}$ is the double  of $\MM$.

\section{A proof of the formula}

It is now customary to formulate Dehn-Sommerville relations in terms of  $h$ - numbers which are the coefficients in the expansion
$F(x-1)=\sum_{n=0}^{d+1} h_n x^{d+1-n}$
where  $F(x)=x^{d+1}+\sum_{n=0}^d f_n x^{d-n}$.
We have  \cite{NS} 
$$h_{d+1-n}-h_n=(-1)^n\binom{d+1}{n}(\chi(\MM)-\chi(\Ss^d)),\,\,\,(0\le n\le d+1).$$
For our purposes it is  convenient to reformulate this result  as follows.
$$$$
\begin{lemma}
If $\MM$ is a semi-Eulerian complex of  dimension $d$ then the polynomial 
$$p(z)=\frac{1}{2}\chi(\MM)+\sum_{n=0}^d f_nz^{n+1}$$
satisfies the identity 
\begin{equation}
p(z)+(-1)^dp(-1-z)=0. \label{l}
\end{equation}
\end{lemma}
{\bf Remark}   The equality  \eqref{l}  has already  appeared in the literature
\cite[Lemma 3]{Kn} but, for some  reasons, it was formulated 
in a way excluding the case when  $d$ is even and  $\chi(\MM)\neq 2$.
\begin{proof}
The above relations can be written in the form 
$$x^{d+1}F\left(\frac{1-x}{x}\right)-F(x-1)=(\chi(\MM)-\chi(\Ss^d))(x-1)^{d+1}.$$
By definition, $p(z)=\frac{1}{2}\chi(\MM)-1+z^{d+1}F(z^{-1})$.
When $d$ is  even  $\chi(\Ss^d)=2$, and taking  $x=1+z^{-1}$ we have 
$$p(z)+p(-1-z)=\chi(\MM)-2+(x-1)^{-d-1}F(x-1)-\left(\frac{x}{x-1}\right)^{d+1}F\left(\frac{1-x}{x}\right)=0.$$
For an odd  $d$ we have $\chi(\MM)=\chi(\Ss^d)=0$ and  the same calculation gives $p(z)-p(-1-z)=0$. 
\end{proof}
\begin{lemma}
Let $\theta: \Rs[z]\to\Rs$ be a linear map defined by
$$\theta(z^n)=\beta_{n-1}\,(n\ge 0).$$ 
If $p\in  \Rs[z]$ is a polynomial 
satisfying   the identity $p(z)+p(-1-z)=0$ then $\theta(p)=0$.
If $q$ is an even polynimial then $\theta(q)=-2q(0)$. 
Moreover, $\theta$ is the only linear functional on  $\Rs[z]$ 
with this  properties. 
\end{lemma}
\begin{proof} The equality $\theta(q)=-2q(0)$ is obvious so we 
only have to prove that $\theta(p)=0$.  Without loss of generality we can assume that $p(z)=(1+2z)^m$ where $m$ is odd. Then 
$$\theta(p)=\sum_{n=0}^m 2^n\binom{m}{n}\theta(z^n)=4\sum_{n=1}^{m+1} \binom{m}{n-1}\frac{2^{n-1}(2^n-1)B_n}{n}=2m!S_m$$
where $S_m=\sum_{n=1}^{m+1} \frac{(4^n-2^n)B_n}{n!(m+1-n)!}$. The number $S_m$ is equal to the coefficient at $z^{m+1}$ in the Taylor 
series of the product 
$$\left(\sum_{n=1}^\infty  \frac{(4^n-2^n)B_n}{n!}z^n\right)\cdot e^z=-\frac{z}{\cosh z}.$$
But this is an odd function so the coefficient at an even power of $z$ is  zero.
Additionally, any polynomial  $P\in \Rs[z]$ can be presented in  the form $P=p+q$ so a functional with the  properties required by the Lemma is unique.
\end{proof}

With this two lemmas the proof of \eqref{t}   becomes obvious. If  $d$ is even then $p(z)+p(-1-z)=0$ for the polynomial 
defined in Lemma 1. Applying $\theta$ to this polynomial we have
$$\theta(p)=-\chi(\MM)+\sum_{n=0}^{d} \beta_n f_n=0.$$

{\bf Remark} If we did not want to compute the coefficients  explicitely then the proof could be a few lines shorter. It is easy  to  show  (e.g. by  induction in the degree)  that a  decomposition  $P=p+q$ from Lemma 2   is  unique. This fact by itself  implies  existence and uniqueness of  $\theta$.

\end{document}